\documentclass[11pt]{article}

\usepackage[table]{xcolor}
\usepackage{amsmath,amsthm,amsfonts,amssymb,amscd,amsxtra,mathrsfs,commath,url}
\usepackage{graphicx}
\usepackage{amsmath,amssymb,mathtools}
\usepackage{algorithm}%
\usepackage{algorithmicx}%
\usepackage{algpseudocode}%
\usepackage{hyperref,caption}
\usepackage{xcolor}
\usepackage{booktabs}
\usepackage{enumitem}
\usepackage{epstopdf}
\usepackage{threeparttable,float}
\usepackage{multirow} 
\usepackage{booktabs} 
\newsavebox\CBox
\def\textBF#1{\sbox\CBox{#1}\resizebox{\wd\CBox}{\ht\CBox}{\textbf{#1}}}
\usepackage{setspace}
\usepackage{cite}

\def\textBF#1{\sbox\CBox{#1}\resizebox{\wd\CBox}{\ht\CBox}{\textbf{#1}}}

\colorlet{mylinkcolor}{red!80!black}
\colorlet{myurlcolor}{green!50!black}
\colorlet{mysectioncolor}{blue!50!black}
\hypersetup{
	colorlinks=true,
	linkcolor=mylinkcolor,
	urlcolor=myurlcolor,
	hypertexnames,
}

\usepackage[doc]{optional}
\usepackage{enumitem}
\usepackage{pstricks}
\usepackage{mathtools,esvect}
\usepackage{cite}
\usepackage{color}
\usepackage{subcaption}
\usepackage{soul}
\definecolor{labelkey}{rgb}{0,0.08,0.45}
\definecolor{refkey}{rgb}{0,0.6,0.0}

\usepackage{float}
\usepackage{threeparttablex,color,soul,colortbl,hhline,rotating,booktabs,longtable,multirow,makecell}
\definecolor{lightgray}{gray}{0.95}

\usepackage[left]{lineno}
\usepackage{blindtext}

\usepackage{fancyhdr}

\usepackage[margin=1.2in]{geometry}

\newcommand{\fenv}[1]%
{\ensuremath{\,\overrightarrow{\operatorname{env}}_{#1}}}
\newcommand{\benv}[1]%
{\ensuremath{\,\overleftarrow{\operatorname{env}}_{#1}}}

{\begin{list}{}{%
\settowidth{\labelwidth}{\textrm{#1~}}%
\setlength{\leftmargin}{\labelwidth+\labelsep}}}
{\end{list}}
\newtheorem{theorem}{Theorem}[section]
\newtheorem{lemma}[theorem]{Lemma}

\newtheorem{proposition}[theorem]{Proposition}
\newtheorem{definition}[theorem]{Definition}
\newtheorem{remark}[theorem]{Remark}

\def\doi{DOI}

\newcounter{count}

\allowdisplaybreaks

\begin{document}

\title{\textrm{Generalized conditional gradient methods for multiobjective composite optimization problems with H{\"o}lder condition}}
\author{
Wang Chen$^{*}$, Liping Tang$^{*}$, Xinmin Yang$^{*}$\thanks{National Center for Applied Mathematics in Chongqing,  Chongqing Normal University, Chongqing, 401331, China. E-mail:
\texttt{chenwangff@163.com; tanglipings@163.com; xmyang@cqnu.edu.cn.}}}

\maketitle

\begin{abstract} 
 \noindent 
In this paper, we deal with multiobjective composite optimization problems, where each objective function is a combination of smooth and  possibly  non-smooth functions. We first propose a parameter-dependent generalized conditional gradient method to solve this problem. The step size in this method requires prior knowledge of the parameters related to the H{\"o}lder continuity of the gradient of the smooth function. The convergence properties of this method are then established. Given that these parameters may be unknown or, if known, may not be unique, the first method may encounter implementation challenges or slow convergence. To address this, we further propose a parameter-free generalized conditional gradient method that determines the step size using a local quadratic upper approximation and an adaptive line search strategy, eliminating the need for any problem-specific parameters. 
The performance of the proposed methods is demonstrated on several test problems involving the indicator function and an uncertainty function.
\end{abstract}

\section{Introduction}
In multiobjective optimization, several objective functions have to be optimized simultaneously under possible constraints. This type of problem has been extensively studied by Jahn \cite{jahn2011vector}, L{\"o}hne \cite{lohne2011vector}, Miettinen \cite{miettinen2012nonlinear} and Yang \cite{yang2018generalized}, and it has found wide applications in various fields (see \cite{RB2013, TI2020}). In contrast with single objective optimization problems, where a single solution optimizes the objective, multiobjective optimization involves conflicting objectives --- improving one objective often leads to the deterioration of another. As a result, there is no single solution that optimizes all objectives simultaneously. Therefore, the optimality notion is replaced with \emph{Pareto optimality}, which represents an optimal compromise among all objectives.

In this paper, we are concerned with the following multiobjective optimization problem:
\begin{equation}\label{mop}
	\begin{aligned}
		&\text{min}\quad F(x)=G(x)+H(x)\\
		&\text{s.t.}\quad\; x\in \mathbb{R}^{n},
	\end{aligned}
\end{equation}
where $F:\mathbb{R}^{n}\rightarrow(\mathbb{R}\cup\{+\infty\})^{m}$, $G:\mathbb{R}^{n}\rightarrow(\mathbb{R}\cup\{+\infty\})^{m}$ and $H:\mathbb{R}^{n}\rightarrow\mathbb{R}^{m}$ are vector-valued functions with $F(x)=(	f_{1}(x),...,f_{m}(x))^{\top}$, $G(x)=(g_{1}(x),...,g_{m}(x))^{\top}$ and  $H(x)=(h_{1}(x),...,h_{m}(x))^{\top}$, respectively. For each $i=1,\ldots,m$, we assume that the function $g_{i}:\mathbb{R}^{n}\rightarrow\mathbb{R}\cup\{+\infty\}$ is proper convex and lower semicontinuous and the function $h_{i}:\mathbb{R}^{n}\rightarrow\mathbb{R}$ is continuously differentiable. When $G(x)$ is the zero function (i.e., $g_{i}=0$ for all $i=1,\ldots,m$) in \eqref{mop},  multiobjective descent algorithms, such as the steepest descent method \cite{fliege2000steepest, cocchi2020convergence}, the Newton method \cite{fliege2009newton}, the quasi-Newton method \cite{qu2011quasi, ansary2015modified}, the trust-region method \cite{carrizo2016trust}, the Barzilai-Borwein method \cite{morovati2016barzilai} and the conjugate gradient method \cite{lucambio2018nonlinear}, are applicable for solving such problem. However, if $G(x)$ is not the zero function, then the aforementioned methods are no longer applicable. We would like to mention that problem \eqref{mop} with $G(x)\neq 0$ involves many important applications. In particular, when $h_{i}$ ($i=1,\ldots,m$) is an indicator function of a convex set  $\Omega\subseteq\mathbb{R}^{n}$, \eqref{mop} reduces to the following constrained multiobjective optimzation problem:
\begin{equation}\label{mop2}
	\begin{aligned}
		&\text{min}\quad H(x)\\
		&\text{s.t.}\quad\; x\in \Omega.
	\end{aligned}
\end{equation}
Furthermore, in approximation theory, problems of
type \eqref{mop} may occur (see \cite{gopfert2003variational,zhao2024convergence}). Additionally, the separable structure \eqref{mop} can be used to model robust multiobjective
optimization problems, which involves handling uncertainty in the data. Further details about applications can be found in \cite{tanabe2019proximal,zhao2024convergence}. 

To the best of our knowledge, the first work solving the separable structure \eqref{mop} was proposed by  Bo{\c{t}} and Grad in \cite{boct2018inertial}. In this work,  two forward--backward proximal point type algorithms with
inertial/memory effects were studied. In \cite{tanabe2019proximal}, Tanabe et al. proposed a proximal gradient method with/without line searches to solve \eqref{mop}. Subsequently, accelerated versions of this method were also explored in \cite{tanabe2023accelerated,nishimura2024monotonicity,tanabe2022globally} to solve \eqref{mop} with convex objective function. Zhao et al. \cite{zhao2024convergence} proposed a new proximal gradient method with a line search procedure, distinct from the Armijo line search used in \cite{tanabe2019proximal}, to determine the iterative step size for solving problem \eqref{mop}.
Bello-Cruz et al. \cite{bello2024proximal} developed a proximal gradient method featuring a novel explicit line search procedure tailored for solving problem \eqref{mop}, where each $g_{i}$ and $h_{i}$ being convex. In \cite{ansary2023newton}, Ansary introduced a globally convergent Newton-type proximal gradient method for solving \eqref{mop}. For solving the single objective optimization problem given by \eqref{mop} with $m=1$, one of the
most widely used iterative algorithms is the conditional gradient method, a.k.a. Frank-Wolfe method (see \cite{bredies2009generalized,rakotomamonjy2015generalized,ito2023parameter,zhao2023analysis,zhao2023generalized,beck2017first}). Recently, Assun{\c{c}}{\~a}o et al. \cite{assunccao2023generalized} proposed a multiobjective version of the scalar conditional gradient method to solve \eqref{mop}. Here, we call it
generalized multiobjective conditional gradient (GM-CondG) method. While they also established asymptotic convergence properties and iteration-complexity bounds of this method with Armijo-type, adaptive and diminishing step sizes. Gebrie and Fukuda \cite{gebrie2024adaptive} introduced an adaptive version of the GM-CondG method, which combines the so-called normalized descent direction as an adaptive procedure and the Armijo-like line search technique. The GM-CondG method has been extended to the Riemannian setting in \cite{li2023generalized}.
It is worth noting that \cite{assunccao2021conditional} was the first to propose a multiobjective conditional gradient (M-CondG) algorithm with Armijo-type, adaptive and diminishing step sizes to solve \eqref{mop2}. For further researches on the M-CondG method for solving \eqref{mop2}, we recommend the readers refer to \cite{chen2024convergence,gonccalves2024improved,2024padhayaynonmonotone,chen2023conditional}. Although significant progress has been made with multiobjective descent algorithms for solving \eqref{mop}, deriving certain desirable convergence properties of the algorithms often requires assuming Lipschitz continuity of the gradient of the smooth objective function (see \cite{boct2018inertial,tanabe2019proximal,assunccao2021conditional,assunccao2023generalized,zhao2024convergence,bello2024proximal}). The smooth function $\nabla h_{i}$ ($i=1,\ldots,m$) has a Lipschitz continuous gradient if there is an $L_{i}>0$ such that $\|\nabla h_{i}(x)-\nabla h_{i}(y)\|\leq L_{i}\|x-y\|$ for all $x,y\in{\rm dom}(g_{i})$. Denote $L=\max_{i=1,\ldots,m} L_{i}$. Additionally, the adaptive step size involved in the GM-CondG and M-CondG methods still relies on the Lipschitz constant $L$. 

This paper aims to consider a more general class of multiobjective optimization problems in which the smooth function has H{\"o}lder continuous gradients, that is,
\begin{enumerate}[label=(A\arabic*), start=1, leftmargin=2.5em]
	\item For each $i=1,\ldots,m$, the gradient $\nabla h_{i}$ is H{\"o}lder continuous on ${\rm dom}(g_{i})$, i.e., there exsit $\nu\in(0,1]$ and $M_{i,\nu}>0$ such that
	$\|\nabla h_{i}(x)-\nabla h_{i}(y)\|\leq M_{i,\nu}\|x-y\|^{\nu}$ for all $x,y\in{\rm dom}(g_{i})$. Denote $M_{v}=\max_{i=1,\ldots,m} M_{i,v}$.\label{holder_gradient}
\end{enumerate}
For more detailed information on H{\"o}lder continuity of vector-valued functions, we recommend the readers refer to the book \cite{fiorenza2017holder}. The particular case of H{\"o}lder continuity with exponent $\nu$ equal to one is generally referred to as Lipschitz continuity. Therefore, the GM-CondG and M-CondG methods equipped with an adaptive step size will no longer be suitable for solving multiobjective optimization problems that satisfy \ref{holder_gradient} with $\nu\neq1$. 

Motivated by the work of \cite{ito2023parameter}, we take a step further on the direction of \cite{assunccao2021conditional,assunccao2023generalized}. Our contribution in this paper is as follows.
\begin{itemize}
	\item To solve \eqref{mop} satisfying \ref{holder_gradient}, we  introduce a new adaptive step size that depends on the parameters $\nu$ and $M_{\nu}$ explicitly, and then propose a parameter-dependent generalized multiobjective conditional gradient (PGM-CondG) method. We analyze the convergence properties of the PGM-CondG method without and with convexity assumptions on the objective functions.
	\item The requirement  for a priori knowledge of the parameters $\nu$ and $M_{\nu}$, associated with the H{\"o}lder continuity of $\nabla h_{i}$, limits the practical implementation of the PGM-CondG method. To address this, we further propose a parameter-free generalized multiobjective conditional gradient (FGM-CondG) method, where the step size is determined  by using a constructive local quadratic upper approximation and an adaptive line search scheme, without needing prior knowledge of $\nu$ and $M_{\nu}$. The convergence properties of the FGM-CondG method  without and with convexity assumptions on the objective functions are also established. 
	\item The performance of the proposed methods in this paper in terms of the metrics --- time, iteration counts and three measures \cite{custodio2011direct}  --- on a few test problems is presented.
\end{itemize}

To our knowledge, this work is the first to provide multiobjective gradient-type algorithms for solving multiobjective composite optimization problems under the H{\"o}lder condition. The remainder of this paper is organized as follows. Section \ref{sec-definitions} presents some notations, definitions and auxiliary results which will be used in the sequel. Section \ref{sec:alg1} introduces a parameter-dependent generalized multiobjective  conditional gradient method and studies its convergence properties. Section \ref{sec:alg2} proposes a parameter-free version of generalized multiobjective  conditional gradient method and obtain its convergence results. Section \ref{num_exp} is devoted to the numerical experiments. Finally, Section \ref{conclusion} draws the conclusions of this paper.

\section{Notations and Preliminaries}\label{sec-definitions}

\subsection{Some notations and basic results}
Throughout this paper, we denote by $\langle \cdot,\cdot\rangle$ and $\|\cdot\|$, respectively, the usual inner product and the Euclidean norm in $\mathbb{R}^{n}$. For any real number $\varrho$, let $\varrho_{+}=\max\{\varrho,0\}$. For any positive integer $m$, let $\langle m\rangle=\{1,\ldots,m\}$. Denote $e=(1,\ldots,1)^{\top}\in\mathbb{R}^{n}$ and $\overline{\mathbb{R}}=\mathbb{R}\cup\{+\infty\}$. 
Let $\mathbb{R}_{+}^{m}$ denote the non-negative orthant of $\mathbb{R}^{m}$. As usual, for $u,v\in\mathbb{R}^{m}$, we use ``$\preceq$'' to denote the classical partial order defined by
\begin{equation*}
	u\preceq v ~\Leftrightarrow ~v-u\in\mathbb{R}_{+}^{n}.
\end{equation*}

The effective domain of $\phi:\mathbb{R}^{n}\rightarrow\overline{\mathbb{R}}$ is defined as ${\rm dom}(\phi)=\{x\in\mathbb{R}^{n}:\phi(x)<+\infty\}$. 
The directional derivative of $\phi:\mathbb{R}^{n}\rightarrow\overline{\mathbb{R}}$ at ${x}\in{\rm dom}(\phi)$ with respect to a vector $d\in\mathbb{R}^{n}$ is defined as 
\begin{equation*}
	\phi'({x};d)=\lim\limits_{\lambda\downarrow0}\dfrac{\phi({x}+\lambda d)-\phi(x)}{\lambda}.
\end{equation*}
It is evident that if $\phi$ is differentiable at $x$, then for any $d\in\mathbb{R}^{n}$, $\phi'({x};d)=\langle\nabla\phi(x),d\rangle$.

The effective domain of $G$ is denoted by ${\rm dom}(G)=\{x\in\mathbb{R}^{n}:g_{i}(x)<\infty,\forall i\in\langle m\rangle\}$. In the subsequent analysis, we will always assume that ${\rm dom}(G)$ is convex and compact. We denoted by $D$ the diameter of ${\rm dom}(G)$, i.e., 
$$D=\sup_{x,y\in{\rm dom}(G)}\|x-y\|.$$
The function $G$ is said to be convex if 
$$G(\lambda x+(1-\lambda)y)\preceq \lambda G(x)+(1-\lambda)G(y)$$
for all $x,y\in{\rm dom}(G)$ and $\lambda\in[0,1]$, or equivalently, if each component $g_{i}$ of $G$, $i\in\langle m\rangle$, is a convex function.

The following lemma will be important for the covergence rate results.

\begin{lemma}\label{lem_aux1}
	{\rm\cite[Lemma 16]{ito2023parameter}} Suppose that $\{\beta_{k}\}$ and $\{\gamma_{k}\}$ are sequences of non-negative real numbers such that 
	\begin{equation}\label{lem_aux1_eq1}
		\gamma_{k+1}\leq\gamma_{k}-c\beta_{k}\min\left\{1,\dfrac{\beta_{k}^{\alpha}}{A}\right\}
	\end{equation}
	for all $k\geq0$ and for some constants $c\in(0,1)$, $\alpha>0$ and $A>0$. Assume additionally that $\beta_{k}\geq\gamma_{k}$ for all $k\geq0$ and let $\beta_{k}^{*}=\min_{0\leq\ell\leq k}\beta_{\ell}$. Then
	$$\gamma_{k}\leq\varGamma_{k}$$
	for all $k\geq k_{0}$ and
		\begin{equation*}
			\beta_{k}^{*}\leq e^{\frac{1}{e}}\varGamma_{\left\lfloor\frac{k+k_{0}+1}{2}\right\rfloor}
		\end{equation*}
	for all $k\geq k_{0}+2A/(c\gamma_{k_{0}}^{\alpha})$,
	where
		\begin{equation*}
			\varGamma_{k}=\left(\dfrac{1}{\gamma_{k_{0}}^{-\alpha}+A^{-1}c\alpha(k-k_{0})}\right)^{\frac{1}{\alpha}},\quad k_{0}=\left\lceil \dfrac{1}{c}\left(\log\dfrac{\gamma_{0}}{cA^{1/\alpha}}\right)_{+}\right\rceil.
		\end{equation*}
\end{lemma}

\begin{lemma}\label{lemma2}
	Suppose that {\rm\ref{holder_gradient}} holds. Then
	\begin{enumerate}[label=\textup{(\roman*)}]
		\item \label{lemma2_item1}
		we have
		\begin{equation}
			H(y)\preceq H(x)+ JH(x)(y-x) + \dfrac{M_{\nu}}{1+\nu}\|y-x\|^{1+\nu}e.
		\end{equation}
		for all $x,y\in{\rm dom}(G)$.
		\item\label{lemma2_item2} it holds that
		\begin{equation}\label{nesterov1}
			H(y)\preceq H(x)+JH(x)(y-x) + \left(\dfrac{L(\varepsilon)}{2}\|y-x\|^{2}+\varepsilon\right)e,
		\end{equation}
		for all $x,y\in{\rm dom}(G)$ and all $\varepsilon>0$, where
		\begin{equation}\label{nesterov2}
			L(\varepsilon)=\left(\dfrac{1-\nu}{1+\nu}\cdot\dfrac{1}{2\varepsilon}\right)^{\frac{1-\nu}{1+\nu}}M_{\nu}^{\frac{2}{1+\nu}},\quad\forall\varepsilon>0.
		\end{equation}
	\end{enumerate}
\end{lemma}

\begin{proof}
	Item \ref{lemma2_item1} follows from \cite[Lemma 1]{yashtini2016global}. Item \ref{lemma2_item2} holds in view of  \cite[Lemma 2]{nesterov2015universal}.
\end{proof}

\subsection{The gap function}

We first review the concepts of (weakly) Pareto optimal points and Pareto stationary points. It is worth noting that the concept of (weakly) Pareto optimality has been widely used in the literature on multiobjective optimization. The Pareto stationarity discussed here was introduced in \cite{tanabe2019proximal}, which generalizes the concept commonly used in \cite{fliege2000steepest}.

\begin{definition}{\rm\cite{miettinen1999nonlinear}}
	A point $\bar{x}\in\mathbb{R}^{n}$ is said to be
	\begin{enumerate}[label=\textup{(\roman*)}]
		\item \emph{Pareto optimal} if there is no $x\in\mathbb{R}^{n}$ such that $F(x)\preceq F(\bar{x})$ and $F(x)\neq F(\bar{x})$;
		\item \emph{weakly Pareto optimal} if there is no $x\in\mathbb{R}^{n}$ such that $f_{i}(x)< f_{i}(\bar{x})$ for all $i\in\langle m\rangle$.
	\end{enumerate}
\end{definition}

\begin{definition}
	{\rm\cite{tanabe2019proximal}}
	A point $\bar{x}\in\mathbb{R}^{n}$ is called a \emph{Pareto stationary} point if for all $d\in\mathbb{R}^{n}$,
	\begin{equation*}
		\max_{i\in\langle m\rangle} f'_{i}(x;d)\geq 0.
	\end{equation*}
\end{definition}

In \cite{assunccao2023generalized}, Assun{\c{c}}{\~a}o et al. intruduced the gap function $\theta:{\rm dom}(G)\rightarrow\mathbb{R}$ associated to \eqref{mop}, defined by
\begin{equation}\label{thetax}
	\theta(x)=\min_{u\in\mathbb{R}^{n}}\max_{i\in\langle m\rangle}\{g_{i}(u)-g_{i}(x)+\langle \nabla h_{i}(x),u-x \rangle\}.
\end{equation}
As noted in \cite{assunccao2023generalized}, the gap function $\theta$ can be used as a stopping criterion for the GM-CondG algorithm. Additionally, when the components of the function $G$ are the indicator function of a set $\Omega$, the gap function $\theta$ reduces to the one proposed in \cite{assunccao2021conditional} 

For every $x\in{\rm dom}(G)$, the gap function $\theta(x)$ is the optimum value of the following optimization problem:
\begin{equation}\label{opt_pro1}
	\min_{u\in\mathbb{R}^{n}}\max_{i\in\langle m\rangle}\{g_{i}(u)-g_{i}(x)+\langle \nabla h_{i}(x),u-x \rangle\}.
\end{equation}
We use the symbol $s(x)\in{\rm dom}(G)$ to denote a solution of  \eqref{opt_pro1}, i.e.,
\begin{equation}\label{opt_sol}
	s(x)\in\mathop{\arg\min}_{u\in\mathbb{R}^{n}}\max_{i\in\langle m\rangle}\{g_{i}(u)-g_{i}(x)+\langle \nabla h_{i}(x),u-x \rangle\}.
\end{equation}
From \eqref{thetax} and \eqref{opt_sol}, we can obtain
\begin{equation}\label{opt_val}
	\theta(x)=\max_{i\in\langle m\rangle}\{g_{i}(s(x))-g_{i}(x)+\langle \nabla h_{i}(x),s(x)-x \rangle\}.
\end{equation}

In the following, we state several properties of the gap function $\theta(x)$, which can be used to characterize the Pareto critical points of \eqref{mop} and to establish a stopping criterion for the algorithm. For more details, see \cite[Lemma 3.1]{assunccao2023generalized}.
\begin{lemma}\label{theta_property}
	Let $\theta:{\rm dom}(G)\rightarrow\mathbb{R}$ be defined as in \eqref{thetax}. Then
	\begin{enumerate}[label=\textup{(\roman*)}]
		\item $\theta(x)\leq0$ for all $x\in{\rm dom}(G)$;
		\item $\theta(x)=0$ if and only if $x$ is a Pareto critical point of \eqref{mop};
		\item $\theta(x)$ is upper semicontinuous.\label{theta_continuous}
	\end{enumerate}
\end{lemma}

We conclude this section by presenting an interesting lemma, which is used to analyze the convergence of the algorithm and degenerates to \cite[Lemma 3.2]{assunccao2023generalized} when $\nu=1$.

\begin{lemma}\label{lemma1}
	Suppose that {\rm\ref{holder_gradient}} holds. Let $x\in{\rm dom}(G)$ and $t\in[0,1]$. Then
	\begin{equation}\label{lemma_eq1}
		F(x+t(s(x)-x))\preceq F(x)+\left(t\theta(x)+\dfrac{M_{\nu}}{1+\nu}t^{1+\nu}\|s(x)-x\|^{1+\nu}\right)e.
	\end{equation}
\end{lemma}

\begin{proof}
	Since \ref{holder_gradient} holds, for all $i\in\langle m\rangle$, we have
	\begin{equation}\label{lemma_eq2}
		\begin{aligned}
			f_{i}(x+t(s(x)-x))&=g_{i}(x+t(s(x)-x))+h_{i}(x+t(s(x)-x))\\
			&\leq g_{i}(x+t(s(x)-x))+h_{i}(x)+t\langle\nabla h_{i}(x),s(x)-x\rangle \\
			&\quad+ \dfrac{M_{\nu}}{1+\nu}t^{1+\nu}\|s(x)-x\|^{1+\nu}.
		\end{aligned}
	\end{equation}
	From the convexity of $g_{i}$, $i\in\langle m\rangle$, we get $g_{i}(x+t(s(x)-x))\leq(1-t)g_{i}(x)+tg_{i}(s(x))$. This, combined with \eqref{lemma_eq2}, yields
	\begin{equation*}\label{lemma_eq3}
		\begin{aligned}
			f_{i}(x+t(s(x)-x))
			&\leq f_{i}(x)+t(g_{i}(s(x))-g_{i}(x)+\langle\nabla h_{i}(x),s(x)-x\rangle)\\
			&\quad+ \dfrac{M_{\nu}}{1+\nu}t^{1+\nu}\|s(x)-x\|^{1+\nu}\\
			&\leq f_{i}(x)+t\max_{i\in\langle m\rangle}\left\{g_{i}(s(x))-g_{i}(x)+\langle\nabla h_{i}(x),s(x)-x\rangle\right\}\\
			&\quad+ \dfrac{M_{\nu}}{1+\nu}t^{1+\nu}\|s(x)-x\|^{1+\nu}\\
			&=f_{i}(x)+t\theta(x)+ \dfrac{M_{\nu}}{1+\nu}t^{1+\nu}\|s(x)-x\|^{1+\nu},\quad i\in\langle m\rangle.
		\end{aligned}
	\end{equation*}
	Thus, the conclusion follows as expected.
\end{proof}

\section{A parameter-dependent generalized multiobjective  conditional gradient method}\label{sec:alg1}

In this section, we introduce a parameter-dependent generalized multiobjective  conditional gradient method to solve \eqref{mop} with the H{\"o}lder continuous gradients. We also discuss the convergence properties of this algorithm. 

In this method, we consider the following step size:
\begin{equation}\label{stp_tk}
	t_{k}=\min\left\{1,\left(\dfrac{|\theta(x^{k})|}{M_{\nu}\|s(x^{k})-x^{k}\|^{1+\nu}}\right)^{\frac{1}{\nu}}\right\},
\end{equation}
which relies on the problem parameters $\nu$ and $M_{\nu}$ explicitly. Observe that, from Lemma \ref{lemma1} and $\theta(x^{k})\leq0$, the step size $t_{k}$ is the minimizer of \begin{equation}\label{qkt}
	q_{k}(t)=-t|\theta(x^{k})|+t^{1+\nu}\dfrac{M_{\nu}\|s(x)-x\|^{1+\nu}}{1+\nu}
\end{equation} 
over $(0,1]$.  Since $\theta(x)<0$ and $s(x)\neq x$ for non-stationary points, the step size is well defined. The algorithm that we propose is given in the following:

\begin{algorithm}[H]
	\caption{}\label{algo1}
	\begin{algorithmic}[1]
		\Require Initialize $x^{0}\in{\rm dom}(G)$.
		\For{$k=0,1,\ldots$}
		\State /* \texttt{Computa the subproblem} */\;
		\State Compute
		\begin{equation}\label{algo1_eq1}
			s(x^{k})\in\mathop{\arg\min}_{u\in\mathbb{R}^{n}}\max_{i\in\langle m\rangle}\left\{g_{i}(u)-g_{i}(x^{k})+\langle \nabla h_{i}(x^{k}),u-x^{k} \rangle\right\},
		\end{equation}
		\begin{equation}\label{algo1_eq2}
			\theta(x^{k})=\max_{i\in\langle m\rangle}\left\{g_{i}(s(x^{k}))-g_{i}(x^{k})+\langle \nabla h_{i}(x^{k}),s(x^{k})-x^{k} \rangle\right\}.
		\end{equation}
		\State /* \texttt{Stopping criteria} */\;
		\If{$\theta(x^{k})=0$}
		\State \textbf{Return}
		$x^{k}$.
		\EndIf
		\State /* \texttt{Compute the step size and iterate}*/\;
		\State Compute $t_{k}$ using \eqref{qkt}.
		\State  Update $x^{k+1}=x^{k}+t^{k}(s(x^{k})-x^{k}).$
		\EndFor
	\end{algorithmic}
\end{algorithm}

\begin{remark}\normalfont
	When $\nu=1$, the step size in \eqref{stp_tk} becomes the adaptive step size proposed in \cite{assunccao2023generalized}, and thus Algorithm \ref{algo1} corresponds to the GM-CondG algorithm with the adaptive step size in \cite{assunccao2023generalized}.
\end{remark}

The next lemma shows that, if $\{x^{k}\}$ is generated by Algorithm \ref{algo1}, then the sequence $\{F(x^{k})\}$ is non-increasing.

\begin{lemma}\label{lema2}
	Let $\{x^{k}\}$ be the sequence generated by Algorithm {\rm\ref{algo1}}. Then, for all $k\geq0$, it holds that
	\begin{equation}\label{lema2_eq0}
		F(x^{k+1})-F(x^{k})\preceq-\dfrac{\nu}{1+\nu}|\theta(x^{k})|\min\left\{1,\left(\dfrac{-\theta(x^{k})}{M_{\nu}\|s(x^{k})-x^{k}\|^{1+\nu}}\right)^{\frac{1}{\nu}}\right\}e.
	\end{equation}
\end{lemma}

\begin{proof}
	Since \ref{holder_gradient} holds, by \eqref{lemma_eq1} invoked with $x = x_{k}$ and $t=t_{k}$, we
	have
	\begin{equation}\label{lema2_eq1}
		F(x^{k+1})\preceq F(x^{k})+\left(t_{k}\theta(x^{k})+\dfrac{M_{\nu}}{1+\nu}t_{k}^{1+\nu}\|s(x^{k})-x^{k}\|^{1+\nu}\right)e.
	\end{equation}
	According to \eqref{stp_tk}, there
	are two options:
	
	\emph{Case 1.} Let $t_{k}=1$. Then, it follows from \eqref{stp_tk} that 
	\begin{equation}\label{lema2_eq2}
		M_{\nu}\|s(x^{k})-x^{k}\|^{1+\nu}\leq |\theta(x^{k})|.
	\end{equation}
	By \eqref{lema2_eq1} and \eqref{lema2_eq2}, and oberve that $|\theta(x^{k})|=-\theta(x^{k})$, we obtain
	\begin{equation}\label{lema2_eq3}
		\begin{aligned}
			F(x^{k+1})&\preceq F(x^{k})+\left(\theta(x^{k})+\dfrac{M_{\nu}}{1+\nu}\|s(x^{k})-x^{k}\|^{1+\nu}\right)e\\
			&=F(x^{k})-\dfrac{\nu}{1+\nu}|\theta(x^{k})|e.
		\end{aligned}
	\end{equation}
	
	\emph{Case 2.} Let $t_{k}=(|\theta(x^{k})|/(M_{\nu}\|s(x^{k})-x^{k}\|^{1+\nu}))^{1/\nu}$.  This,
	together with \eqref{lema2_eq1}, yields
	\begin{equation}\label{lema2_eq4}
		\begin{aligned}
			F(x^{k+1})&\preceq F(x^{k})-\dfrac{\nu}{1+\nu}|\theta(x^{k})|\left(\dfrac{|\theta(x^{k})|}{M_{\nu}\|s(x^{k})-x^{k}\|^{1+\nu}}\right)^{\frac{1}{\nu}}.
		\end{aligned}
	\end{equation}
	Therefore, \eqref{lema2_eq0} is directly derived by \eqref{lema2_eq3} and \eqref{lema2_eq4}.
\end{proof}

Define
\begin{equation}
	f_{0}^{\max}=\max_{i\in\langle m\rangle}\{f_{i}(x^{0})\} \quad{\rm and}\quad
	f^{\inf}=\min_{i\in\langle m\rangle}\{f_{i}^{*}\},
\end{equation}
where $f_{i}^{*}=\inf\{f_{i}(x):x\in{\rm dom}(G)\}$ for all $i\in\langle m\rangle$.

The following result demonstrates the convergence of the gap function to zero with respect to the iterative sequence generated by the algorithm, as well as the characteristics of the convergence rate of the gap function.

\begin{theorem}\label{thm1}
	Let $\{x^{k}\}$ be the sequence generated by Algorithm {\rm\ref{algo1}}. Then
	\begin{enumerate}[label=\textup{(\roman*)}]
		\item $\lim_{k\rightarrow\infty}\theta(x^{k})=0$;\label{thm1_i}
		\item the seqence $\{\theta(x^{\ell})\}$ satisfies
		\begin{equation*}
			\min_{0\leq\ell\leq k}|\theta(x^{j})| \leq\max\left\{\dfrac{(1+\nu)(f_{0}^{\max}-f^{\inf})}{\nu(k+1)},\left(\dfrac{(1+\nu)M_{\nu}^{\frac{1}{\nu}}D^{\frac{1+\nu}{\nu}}(f_{0}^{\max}-f^{\inf})}{\nu(k+1)}\right)^{\frac{\nu}{1+\nu}}\right\}.
		\end{equation*}
	\end{enumerate}
\end{theorem}

\begin{proof}
	Let us show item (i). Observe that $x^{k}$ and $s(x^{k})$ are contained in ${\rm dom}(G)$ for all $k\geq0$. It then follows that $\|s(x^{k})-x^{k}\|\leq D$. By this and \eqref{lema2_eq0}, we obtain
	\begin{equation}\label{thm1_eq1}
		0\preceq\dfrac{\nu}{1+\nu}|\theta(x^{k})|\min\left\{1,\left(\dfrac{-\theta(x^{k})}{M_{\nu}D^{1+\nu}}\right)^{\frac{1}{\nu}}\right\}e\preceq F(x^{k})-F(x^{k+1}).
	\end{equation}
	This implies that $\{F(x^{k})\}$ is non-increasing, which, together with the fact that $\{f_{i}(x^{k})\}$ is bounded from below by $f^{\inf}$ for all $i\in\langle m\rangle$, gives that the sequence $\{F(x^{k})\}$  is convergent. This obviously means that
	\begin{equation}\label{thm1_eq2}
		\lim_{k\rightarrow\infty}(F(x^{k+1})-F(x^{k}))=0.
	\end{equation}
	Taking $\lim_{k\rightarrow\infty}$ in \eqref{thm1_eq1}, and then combining with \eqref{thm1_eq2}, we have $\lim_{k\rightarrow\infty}\theta(x^{k})=0$.
	
	Now consider item (ii). From \eqref{thm1_eq1}, it follows that
	\begin{equation*}
		\sum_{j=0}^{k}|\theta(x^{j})|\min\left\{1,\left(\dfrac{-\theta(x^{j})}{M_{\nu}D^{1+\nu}}\right)^{\frac{1}{\nu}}\right\}\leq\dfrac{(1+\nu)(f_{0}^{\max}-f^{\inf})}{\nu}.
	\end{equation*}
	Denote $E_{k}=(-\theta(x^{j})/(M_{\nu}D^{1+\nu}))^\frac{1}{\nu}$. If $E_{k}>1$, then
	\begin{equation}\label{thm1_eq3}
			\dfrac{1}{k+1}\min_{0\leq\ell\leq k}|\theta(x^{j})|\leq\sum_{j=0}^{k}|\theta(x^{j})|\leq\dfrac{(1+\nu)(f_{0}^{\max}-f^{\inf})}{\nu}.
	\end{equation}
	If $E_{k}\leq1$, then
	\begin{equation*}
		\sum_{j=0}^{k}|\theta(x^{j})|^{\frac{1+\nu}{\nu}}\leq\dfrac{(1+\nu)M_{\nu}^{\frac{1}{\nu}}D^{\frac{1+\nu}{\nu}}(f_{0}^{\max}-f^{\inf})}{\nu},
	\end{equation*}
	which, combined with the fact that 
	\begin{equation*}
			\dfrac{1}{k+1}\min_{0\leq\ell\leq k}|\theta(x^{j})|^{\frac{1+\nu}{\nu}}\leq\sum_{j=0}^{k}|\theta(x^{j})|^{\frac{1+\nu}{\nu}}
	\end{equation*}
	gives
	\begin{equation}\label{thm1_eq4}
	\min_{0\leq\ell\leq k}|\theta(x^{j})|\leq\left(\dfrac{(1+\nu)M_{\nu}^{\frac{1}{\nu}}D^{\frac{1+\nu}{\nu}}(f_{0}^{\max}-f^{\inf})}{\nu(k+1)}\right)^{\frac{\nu}{1+\nu}}.
	\end{equation}
	The combination of \eqref{thm1_eq3} and \eqref{thm1_eq4} implies the conclusion of item (ii).
\end{proof}

\begin{theorem}\label{alg1_convergence}
	Let $\{x^{k}\}$ be the sequence generated by Algorithm {\rm\ref{algo1}}. Then, every limit point $\{x^{k}\}$ is a Pareto critical point of \eqref{mop}.
\end{theorem}

\begin{proof}
	The proof follows from Lemma \ref{theta_property} and Theorem \ref{thm1}\ref{thm1_i}.
\end{proof}

In the following, we will discuss the convergence rate of Algorithm \ref{algo1} with respect to the gap function and the objective function value, under the assumption that the function $H$ is convex.
Denote 
$$\delta_{k}(x^{*})=\min_{i\in\langle m\rangle}\{F_{i}(x^{k})-F_{i}(x^{*})\}.$$

\begin{theorem}\label{thm2}
	Assume that $H$ is convex. Let $\{x^{k}\}$ be the sequence generated by Algorithm {\rm\ref{algo1}}. Assume that there exists $x^{*}\in{\rm dom}(G)$ such that $F(x^{*})\preceq F(x^{k})$ for all $k$. 
	Then, we have 
	$$\delta_{k}(x^{*})\leq \varGamma_{k}$$
	for $k\geq k_{0}$, where 
	$$\varGamma_{k}=\left(\dfrac{1}{(\delta_{k_{0}}(x^{*}))^{-\alpha}+A^{-1}c\alpha(k-k_{0})}\right)^{\frac{1}{\alpha}}, \quad k_{0}=\left\lceil \dfrac{1}{c}\left(\log\dfrac{\gamma_{0}}{cA^{1/\alpha}}\right)_{+}\right\rceil,$$
	\begin{equation*}
		c=\dfrac{\nu}{1+\nu},\quad \alpha=\dfrac{1}{\nu}\quad{\rm and}\quad A=M_{\nu}^{\frac{1}{\nu}}D^{\frac{1+\nu}{\nu}}.
	\end{equation*}
	Furthermore,
	$$
	\min_{0\leq j\leq k}|\theta(x^{k})|\leq e^{\frac{1}{e}}\varGamma_{\left\lfloor(k+k_{0}+1)/2\right\rfloor}
	$$
	for all $k\geq k_{0}+2A/(c(\delta_{k_{0}}(x^{*}))^{\alpha})$.
\end{theorem}

\begin{proof}
	According to the relation \eqref{thm1_eq1}, we have
	\begin{equation*}
		\min_{i\in\langle m\rangle}(F_{i}(x^{k+1})-F_{i}(x^{*}))\leq\min_{i\in\langle m\rangle}(F(x^{k})-F_{i}(x^{*}))-\dfrac{\nu}{1+\nu}|\theta(x^{k})|\min\left\{1,\left(\dfrac{-\theta(x^{k})}{M_{\nu}D^{1+\nu}}\right)^{\frac{1}{\nu}}\right\},
	\end{equation*}
	i.e.,
	\begin{equation}\label{thm2_eq1}
		\delta_{k+1}(x^{*})\leq\delta_{k}(x^{*})-\dfrac{\nu}{1+\nu}|\theta(x^{k})|\min\left\{1,\left(\dfrac{-\theta(x^{k})}{M_{\nu}D^{1+\nu}}\right)^{\frac{1}{\nu}}\right\} 
	\end{equation}
	On the other hand, the convexity of $H$ implies that $h_{i}(x^{*})-h_{i}(x^{k})\geq\langle \nabla h_{i}(x^{*}),x^{*}-x^{k}\rangle$, and observe that $F(x^{*})\preceq F(x^{k})$ for all $k$, we have
	\begin{equation*}
		0\geq f_{i}(x^{*})-f_{i}(x^{k})\geq\langle\nabla h_{i}(x^{k}),x^{*}-x^{k}\rangle+g_{i}(x^{*})-g_{i}(x^{k})
	\end{equation*}
	for all $i\in\langle m\rangle$, which implies that
	\begin{equation*}
		\begin{aligned}
			0&\geq \max_{i\in\langle m\rangle}
			(f_{i}(x^{*})-f_{i}(x^{k}))\\
			&\geq\max_{i\in\langle m\rangle}\{\langle\nabla h_{i}(x^{k}),x^{*}-x^{k}\rangle+g_{i}(x^{*})-g_{i}(x^{k})\}\\
			&\geq\max_{i\in\langle m\rangle}\{\langle\nabla h_{i}(x^{k}),s(x^{k})-x^{k}\rangle+g_{i}(s(x^{k}))-g_{i}(x^{k})\}\\
			&=\theta(x^{k}),
		\end{aligned}
	\end{equation*}
	where the third inequality follows from the optimality of $s(x^{k})$ in \eqref{algo1_eq1}. Therefore, observe that $\theta(x^{k})\leq0$, we get
	\begin{equation*}
		0\leq\min_{i\in\langle m\rangle}
		\left\{f_{i}(x^{k})-f_{i}(x^{*})\right\}=\delta_{k}(x^{*})\leq|\theta(x^{k})|.
	\end{equation*}
	By Lemma \ref{lem_aux1} with $\gamma_{k}=\delta_{k}(x^{*})$ and $\beta_{k}=|\theta(x^{k})|$, we can obtain the desired results.
\end{proof}

\section{A parameter-free multiobjective generalized conditional gradient method}\label{sec:alg2}

The step size $t_{k}>0$ in Algorithm \ref{algo1}
might present some practical challenges. To compute such $t_{k}$, it is necessary to know the parameters $\nu$ and $M_{\nu}$ in advance. However, as seen in  \ref{holder_gradient},  it is often difficult to compute these parameters if $h_{i}$ is complex. Note that $t_{k}$ in Algorithm \ref{algo1} is the minimizer of $q_{k}(t)$ as in \eqref{qkt} over $(0,1]$. If the parameters $\nu$ and $M_{\nu}$ are unknown, then we can find a quadratic approximation to \eqref{qkt}, and then minimize it over $(0,1]$ to get a step size. We describe this procedure as follows. 

For any $t\in(0,1]$, using  the convexity of $g_{i}$, the relation \eqref{nesterov1} with $x=x^{k}$, $y=x^{k}+t(s(x^{k})-x^{k})$ and $\varepsilon=t|\theta(x^{k})|/2$, as well as the definition of $\theta(x^{k})$ as in \eqref{thetax}, we have
\begin{equation}\label{solv_tk}
	\begin{aligned}
		f_{i}(x^{k}+t(s(x^{k})-x^{k}))&=g_{i}(x^{k}+t(s(x^{k})-x^{k}))+h_{i}(x^{k}+t(s(x^{k})-x^{k}))\\
		&\leq (1-t)g_{i}(x^{k})+tg_{i}(s(x^{k}))+h_{i}(x^{k}+t(s(x^{k})-x^{k}))\\
		&\leq (1-t)g_{i}(x^{k})+tg_{i}(s(x^{k}))+h_{i}(x^{k})+t\langle\nabla h_{i}(x^{k}), s(x^{k})-x^{k}\rangle\\
		&\quad+\dfrac{L(t|\theta(x^{k})|/2)}{2}t^{2}\|s(x^{k})-x^{k}\|^{2}+\dfrac{t|\theta(x^{k})|}{2}\\
		&= f_{i}(x^{k})+t(g_{i}(s(x^{k}))-g_{i}(x^{k})+\langle\nabla h_{i}(x^{k}), s(x^{k})-x^{k}\rangle)\\
		&\quad+\dfrac{L(t|\theta(x^{k})|/2)}{2}t^{2}\|s(x^{k})-x^{k}\|^{2}+\dfrac{t|\theta(x^{k})|}{2}\\
		&\leq f_{i}(x^{k})+t \theta(x^{k})+\dfrac{L(t|\theta(x^{k})|/2)}{2}t^{2}\|s(x^{k})-x^{k}\|^{2}+\dfrac{t|\theta(x^{k})|}{2}\\
		&=f_{i}(x^{k})-\dfrac{t|\theta(x^{k})|}{2}+\dfrac{L(t|\theta(x^{k})|/2)}{2}t^{2}\|s(x^{k})-x^{k}\|^{2},\quad i\in\langle m\rangle.
	\end{aligned}
\end{equation}
Denote
\begin{equation}\label{qkt1}
	\tilde{q}_{k}(t)=-\dfrac{t|\theta(x^{k})|}{2}+\dfrac{L_{k}}{2}t^{2}\|s(x^{k})-x^{k}\|^{2},
\end{equation} 
for some $L_{k}>0$. Therefore, the step size can be obtained by minimize $\tilde{q}_{k}(t)$ over $(0,1]$. The minimizer is 
$$t_{k}=\min\left\{1,\dfrac{|\theta(x^{k}|}{2L_{k}\|s(x^{k})-x^{k}\|^{2}}\right\}.$$
which eliminates the need for prior knowledge of $\nu$ and $M_{\nu}$. As for the calculation of $L_{k}$, we gives an adaptive line search scheme.

In this following, we propose a parameter-free conditional gradient algorithm (Algorithm \ref{algo2}) for solving \eqref{mop}.

\begin{algorithm}[H]
	\caption{}\label{algo2}
	\begin{algorithmic}[1]
		\Require Initialize $x^{0}\in{\rm dom}(G)$ and $L_{-1}>0$.
		\For{$k=0,1,\ldots$}
		\State /* \texttt{Computa the subproblem} */\;
		\State Compute
		\begin{equation}\label{algo2_eq1}
			s(x^{k})\in\mathop{\arg\min}_{u\in\mathbb{R}^{n}}\max_{i\in\langle m\rangle}\left\{g_{i}(u)-g_{i}(x^{k})+\langle \nabla h_{i}(x^{k}),u-x^{k} \rangle\right\},
		\end{equation}
		\begin{equation}\label{algo2_eq2}
			\theta(x^{k})=\max_{i\in\langle m\rangle}\left\{g_{i}(s(x^{k}))-g_{i}(x^{k})+\langle \nabla h_{i}(x^{k}),s(x^{k})-x^{k} \rangle\right\}.
		\end{equation}
		\State /* \texttt{Stopping criteria} */\;
		\If{$\theta(x^{k})=0$}
		\State \textbf{Return}
		$x^{k}$.
		\EndIf
		\State /* \texttt{The adaptive line search procedure} */\;
		\State \textbf{repeat~for} $\ell=0,1,\ldots$
		\begin{equation*}
			\begin{aligned} 				L_{k}^{\ell}&=2^{\ell-1}L_{k-1},\\
				t_{k}^{\ell}&=\min\left\{1,\dfrac{|\theta(x^{k})|}{2L_{k}^{\ell}\|s(x^{k})-x^{k}\|^{2}}\right\},\\
				x^{k+1}_{\ell}&=x^{k}+t_{\ell}^{k}(s(x^{k})-x^{k}).
			\end{aligned}
		\end{equation*}
		\State\textbf{until}
		\begin{equation}\label{L_linesearch}
			F(x^{k+1}_{\ell})\preceq F(x^{k})+\left(-\dfrac{1}{2}t_{k}^{\ell}|\theta(x^{k})|+\dfrac{1}{2}L_{k}^{\ell}(t_{k}^{\ell})^{2}\|s(x^{k})-x^{k}\|^{2}\right)e.
		\end{equation}
		\State /* \texttt{Update the iterate} */\;
		\State Set $(x^{k+1},L_{k},t_{k})\leftarrow(x_{\ell}^{k+1},L_{k}^{\ell},t_{k}^{\ell})$.
		\EndFor
	\end{algorithmic}
\end{algorithm}

The following propositions show that the adaptive line search procedure can terminate after a finite number of trials.

\begin{proposition}\label{thm4}
	The relation \eqref{L_linesearch} holds whenever $L_{k}^{\ell}\geq \tilde{L}_{k}$, where
	\begin{equation}\label{title_Lk}
		\tilde{L}_{k}=\max\left\{L\left(\dfrac{|\theta(x^{k})|}{2}\right),\left[L\left(\dfrac{|\theta(x^{k})|^{2}}{4\|s(x^{k})-x^{k}\|^{2}}\right)\right]^{\frac{1+\nu}{2\nu}}\right\}
	\end{equation}
	and $L(\cdot)$ is defined in \eqref{nesterov2}.
\end{proposition}

\begin{proof}
	If we set $t=t_{k}^{\ell}$ in \eqref{solv_tk}, then
	\begin{equation}\label{thm4_eq1}
		\begin{aligned}
			f_{i}(x_{\ell}^{k+1})\leq f_{i}(x^{k})-\dfrac{t_{k}^{\ell}|\theta(x^{k})|
			}{2}+\dfrac{L(t_{k}^{\ell}|\theta(x^{k})|/2)}{2}(t_{k}^{\ell})^{2}\|s(x^{k})-x^{k}\|^{2},\quad i\in\langle m\rangle.
		\end{aligned}
	\end{equation}
	Hence, \eqref{L_linesearch} holds if
	\begin{equation}\label{thm4_eq2}
		L_{k}^{\ell}\geq L\left(\dfrac{t_{k}^{\ell}|\theta(x^{k})|}{2}\right)=\max\left\{L\left(\frac{|\theta(x^{k})|}{2}\right),L\left(\dfrac{|\theta(x^{k})|^{2}}{4L_{k}^{\ell}\|s(x^{k})-x^{k}\|^{2}}\right)\right\}
	\end{equation}
	where the equality holds in view of the definitions of $t_{k}^{\ell}$ and \eqref{nesterov2}. Again, from \eqref{nesterov2}, we have
	\begin{equation*}
		\begin{aligned}
			L\left(\dfrac{|\theta(x^{k})|^{2}}{4L_{k}^{\ell}\|s(x^{k})-x^{k}\|^{2}}\right)&=\left(\dfrac{1-\nu}{1+\nu}\cdot\dfrac{1}{2}\cdot\dfrac{4\|s(x^{k})-x^{k}\|^{2}}{|\theta(x^{k})|^{2}}\right)^{\frac{1-\nu}{1+\nu}}(L_{k}^{\ell})^{\frac{1-\nu}{1+\nu}}M_{\nu}^{\frac{2}{1+\nu}}\\	&=(L_{k}^{\ell})^{\frac{1-\nu}{1+\nu}}\left[L\left(\dfrac{|\theta(x^{k})|^{2}}{4\|s(x^{k})-x^{k}\|^{2}}\right)\right]^{\frac{1+\nu}{2\nu}}
		\end{aligned}
	\end{equation*}
	which, combined with \eqref{thm4_eq2}, gives
	\begin{equation*}
		\begin{aligned}
			L_{k}^{\ell}\geq(L_{k}^{\ell})^{\frac{1-\nu}{1+\nu}}\left[L\left(\dfrac{|\theta(x^{k})|^{2}}{4\|s(x^{k})-x^{k}\|^{2}}\right)\right]^{\frac{1+\nu}{2\nu}}
		\end{aligned}
	\end{equation*}
	and thus
	\begin{equation*}
		L_{k}^{\ell}\geq \left[L\left(\dfrac{|\theta(x^{k})|^{2}}{4\|s(x^{k})-x^{k}\|^{2}}\right)\right]^{\frac{1+\nu}{2\nu}}.
	\end{equation*}
	Let 
	$$\tilde{L}_{k}=\max\left\{L\left(\dfrac{|\theta(x^{k})|}{2}\right),\left[L\left(\dfrac{|\theta(x^{k})|^{2}}{4\|s(x^{k})-x^{k}\|^{2}}\right)\right]^{\frac{1+\nu}{2\nu}}\right\}.$$
	Then, \eqref{L_linesearch} holds if $L_{k}^{\ell}\geq \tilde{L}_{k}$.
\end{proof}

\begin{proposition}\label{alg2_proposition2}
	Let $\tilde{L}_{k}$ be as in \eqref{title_Lk}. Then, $L_{k}\leq 2\max_{0\leq j\leq k}\tilde{L}_{j}$ for all $k\geq (\log_{2}(L_{-1}/\tilde{L}_{0}))_{+}$.
\end{proposition}

\begin{proof}
	The proof can follow from the one in \cite[Theorem 9(ii)]{ito2023parameter}.
\end{proof}

\begin{proposition}\label{alg2_proposition3}
	For any $\xi>0$, let
	\begin{equation}\label{barL_xi}
		\bar{L}(\xi)=\max\left\{\left(\dfrac{1-\nu}{1+\nu}\cdot\dfrac{1}{\xi}\right)^{\frac{1-\nu}{1+\nu}}M_{\nu}^{\frac{2}{1+\nu}},\left(\dfrac{2(1-\nu)}{1+\nu}\right)^{\frac{1-\nu}{2\nu}}M_{\nu}^{\frac{1}{\nu}}\left(\dfrac{D}{\xi}\right)^{\frac{1-\nu}{\nu}}\right\}.
	\end{equation}
	Suppose that $\min_{0\leq j\leq k}|\theta(x^{j})|\geq \xi$ for some $k\geq 0$ and $\xi>0$. Then, the total number of inner iterations performed by the adaptive line search procedure until the $k$-th iteration of Algorithm {\rm\ref{algo2}} is bouned by $2k+2+(\log_{2}(2\bar{L}(\xi))/L_{-1})_{+}$.
\end{proposition}

\begin{proof}
	From \eqref{nesterov2} and \eqref{title_Lk}, we have
	$$\tilde{L}_{k}=\max\left\{\left(\dfrac{1-\nu}{1+\nu}\cdot\dfrac{1}{|\theta(x^{k})|}\right)^{\frac{1-\nu}{1+\nu}}M_{\nu}^{\frac{2}{1+\nu}},\left(\dfrac{2(1-\nu)}{1+\nu}\right)^{\frac{1-\nu}{2\nu}}M_{\nu}^{\frac{1}{\nu}}\left(\dfrac{\|s(x^{k})-x^{k}\|}{|\theta(x^{k})|}\right)^{\frac{1-\nu}{\nu}}\right\}.$$
	Observe that $x^{k},s(x^{k})\in{\rm dom}(G)$ for all $k\geq0$. It then follows that $\|v(x^{k})-x^{k}\|\leq D$. By this and the above inequality, we have 
	\begin{equation}\label{titleL_leq_barL}
		\tilde{L}_{k}\leq\bar{L}(|\theta(x^{k})|).
	\end{equation}
	We omit the remaining proof here as it can follow the proof of \cite[Theorem 9(iii)]{ito2023parameter}.
\end{proof}

The following lemma demonstrates that if the sequence $\{x^k\}$ is generated by Algorithm \ref{algo2}, then $\{F(x^k)\}$ is non-increasing.

\begin{lemma}\label{alg2_lema2}
	Let $\{x^{k}\}$ be the sequence generated by Algorithm {\rm\ref{algo2}}. Then, for all $k\geq0$, it holds that
	\begin{equation}\label{alg2_lema2_eq1}
		F(x^{k+1})-F(x^{k})\preceq-\dfrac{|\theta(x^{k})|}{4}\min\left\{1,\dfrac{-\theta(x^{k})}{L_{k}\|s(x^{k})-x^{k}\|^{2}}\right\}e.
	\end{equation}
\end{lemma}

\begin{proof}
	Observe from Algoritm {\rm\ref{algo2}} that $t_{k}=\min\{1,|\theta(x^{k})|/(2L_{k}\|s(x^{k})-x^{k}\|^{2})\}$ and
	\begin{equation*}
		F(x^{k+1})\preceq F(x^{k})+\left(-\dfrac{1}{2}t_{k}|\theta(x^{k})|+\dfrac{1}{2}L_{k}(t_{k})^{2}\|s(x^{k})-x^{k}\|^{2}\right)e.
	\end{equation*}	
	Then, the rest proof goes along the lines of the one of Lemma \ref{lema2}.
\end{proof}

\begin{lemma}\label{alg2_thm1}
	Let $\{x^{k}\}$ be the sequence generated by Algorithm \ref{algo2}. Then, there holds
	\begin{equation}\label{alg2_thm1_eq1}
		F(x^{k+1})-F(x^{k})\preceq-\dfrac{\min_{0\leq j\leq k}|\theta(x^{j})|}{4}\min\left\{1,\left(\dfrac{\min_{0\leq j\leq k}|\theta(x^{j})|}{2M_{\nu}D^{1+\nu}}\right)^{\frac{1}{\nu}}\right\}e
	\end{equation}
	for any $k\geq \tilde{k}_{0}=\lceil(\log_{2}(L_{-1}/\tilde{L}_{0}))_{+}\rceil$.
\end{lemma}

\begin{proof}
	According to Proposition \ref{alg2_proposition2}, the relations \eqref{barL_xi} and \eqref{titleL_leq_barL}, and the fact that $1/\min_{0\leq j\leq k}|\theta(x^{j})|=\max_{0\leq j\leq k}(1/|\theta(x^{j})|)$, we have 
	\begin{equation*}
		L_{k}\leq 2\max_{0\leq j\leq k}\tilde{L}_{j}\leq 2\max_{0\leq j\leq k}\bar{L}(|\theta(x^{j})|)=2\bar{L}\left(\min_{0\leq j\leq k}|\theta(x^{j})|\right)
	\end{equation*}		
	for all $k\geq \tilde{k}_{0}$. This, combined with \eqref{alg2_lema2_eq1}, yields
	\begin{equation}\label{alg2_thm1_eq2}
		F(x^{k+1})-F(x^{k})\preceq-\dfrac{\min_{0\leq j\leq k}|\theta(x^{j})|}{4}\min\left\{1,E_{k}\right\}e,
	\end{equation}
	where
	\begin{equation}\label{Ek}
		E_{k}=\dfrac{1}{2\bar{L}(\min_{0\leq j\leq k}|\theta(x^{j})|)}\min_{0\leq j\leq k}\dfrac{|\theta(x^{j})|}{\|s(x^{j})-x^{j}\|^{2}}
	\end{equation}
	Observe that $x^{k},v(x^{k})\in{\rm dom}(G)$ for all $k\geq0$. It then follows that $\|v(x^{k})-x^{k}\|\leq D$. Clearly,
	for all $j\in[0,k]$, there holds
	\begin{equation*}
		\dfrac{|\theta(x^{j})|}{\|v(x^{j})-x^{j}\|}\geq\dfrac{\min_{0\leq j\leq k}|\theta(x^{j})|}{D^{2}}.
	\end{equation*}
	Hence, 
	\begin{equation*}
		E_{k}\geq\dfrac{\min_{0\leq j\leq k}|\theta(x^{j})|}{2D^{2}\bar{L}(\min_{0\leq j\leq k}|\theta(x^{j})|)}
	\end{equation*}
	According to the definition of $\bar{L}$ in \eqref{barL_xi}, there are two options:
	
	\emph{Case 1.} Let
	\begin{equation*}
		\bar{L}(\xi)=\left(\dfrac{2(1-\nu)}{1+\nu}\right)^{\frac{1-\nu}{2\nu}}M_{\nu}^{\frac{1}{\nu}}\left(\dfrac{D}{\xi}\right)^{\frac{1-\nu}{\nu}}.
	\end{equation*}
	Then, we have
	$$E^{k}\geq 2^{-\frac{1+\nu}{2\nu}}\left(\dfrac{1-\nu}{1+\nu}\right)^{-\frac{1-\nu}{2\nu}}\left(\dfrac{\min_{0\leq j\leq k}|\theta(x^{j})|}{M_{\nu}D^{1+\nu}}\right)^{\frac{1}{\nu}}.$$
	For simplicity of notation, we denote the right-hand-side of the above inequality as $Q_{k}$.
	
	\emph{Case 2.} Let
	\begin{equation*}
		\bar{L}(\xi)=\left(\dfrac{1-\nu}{1+\nu}\cdot\dfrac{1}{\xi}\right)^{\frac{1-\nu}{1+\nu}}M_{\nu}^{\frac{2}{1+\nu}}.
	\end{equation*}
	Then, we obtain
	$$E_{k}\geq\dfrac{1}{2}\left(\dfrac{1-\nu}{1+\nu}\right)^{\frac{1-\nu}{1+\nu}}\left(\dfrac{\min_{0\leq j\leq k}|\theta(x^{k})|}{M_{\nu}D^{1+\nu}}\right)^{\frac{2}{1+\nu}}=Q_{k}^{\frac{2\nu}{1+\nu}}.$$
	Therefore, $E_{k}\geq\min\{Q_{k},Q_{k}^{2\nu/(1+\nu)}\}$. This, combined with the fact that $2\nu/(1+\nu)\leq1$, gives $\min\{1,E_{k}\}\geq\min\{1,Q^{k}\}$. By this, \eqref{alg2_thm1_eq2}, \eqref{Ek} and the definition of $Q_{k}$, we obtain
	\begin{equation*}
		\begin{aligned}
			F(x^{k+1})-F(x^{k})&\preceq-\dfrac{\min_{0\leq j\leq k}|\theta(x^{j})|}{4}\min\left\{1,Q_{k}\right\}e\\
			&\preceq-\dfrac{\min_{0\leq j\leq k}|\theta(x^{j})|}{4}\min\left\{1,\left(\dfrac{\min_{0\leq j\leq k}|\theta(x^{j})|}{2M_{\nu}D^{1+\nu}}\right)^{\frac{1}{\nu}}\right\}e,
		\end{aligned}
	\end{equation*}
	which concludes the proof.
\end{proof}

\begin{theorem}\label{alg2_thm2}
	Let $\{x^{k}\}$ be the sequence generated by Algorithm {\rm\ref{algo2}}. Then, we have
	\begin{equation*}
		\min_{0\leq j\leq k}|\theta(x^{j})| \leq\max\left\{\dfrac{4(f_{0}^{\max}-f^{\inf})}{k+1-\tilde{k}_{0}},\left(\dfrac{2^{2+\frac{1}{\nu}M_{\nu}^{\frac{1}{\nu}}D^{\frac{1+\nu}{\nu}}}(f_{0}^{\max}-f^{\inf})}{k+1-\tilde{k}_{0}}\right)^{\frac{\nu}{1+\nu}}\right\},\quad\forall k\geq\tilde{k}_{0},
	\end{equation*}
	where $\tilde{k}_{0}=\lceil(\log_{2}(L_{-1}/\tilde{L}_{0}))_{+}\rceil$.
\end{theorem}

\begin{proof}
	By \eqref{alg2_thm1_eq1}, for all $i\in\langle m\rangle$ and $k\geq\tilde{k}_{0}$, we have
	\begin{equation*}
		\dfrac{\min_{0\leq j\leq k}|\theta(x^{j})|}{4}\min\left\{1,\left(\dfrac{\min_{0\leq j\leq k}|\theta(x^{j})|}{2M_{\nu}D^{1+\nu}}\right)^{\frac{1}{\nu}}\right\}\leq f_{i}(x^{k})-f_{i}(x^{k+1}).
	\end{equation*}
	Hence, we have that $\{F(x^{k})\}$ is monotone decreasing and that
	\begin{equation*}
		\sum_{\ell=0}^{k}\min_{0\leq j\leq \ell}|\theta(x^{j})|\min\left\{1,\left(\dfrac{\min_{0\leq j\leq \ell}|\theta(x^{j})|}{2M_{\nu}D^{1+\nu}}\right)^{\frac{1}{\nu}}\right\}\leq 4(f_{0}^{\max}-f^{\inf}),\quad  k\geq\tilde{k}_{0}.
	\end{equation*}
	Denote $\vartheta=\min\left\{1,\left(\dfrac{\min_{0\leq j\leq \ell}|\theta(x^{j})|}{2M_{\nu}D^{1+\nu}}\right)^{\frac{1}{\nu}}\right\}$, where $\ell\in[0,k]$. Now, we consider two options:
	
	\emph{Case 1.} Let $\vartheta=1$. Then, we obtain
	$
	\sum_{\ell=0}^{k}\min_{0\leq j\leq \ell}|\theta(x^{j})|\leq 4(f_{0}^{\max}-f^{\inf}),
	$
	which implies that
	\begin{equation}
		\min_{0\leq j\leq k}|\theta(x^{j})|\leq\dfrac{4(f_{0}^{\max}-f^{\inf})}{k+1-\tilde{k}_{0}}.
	\end{equation}
	
	\emph{Case 2.} Let $\vartheta=\left(\dfrac{\min_{0\leq j\leq \ell}|\theta(x^{j})|}{2M_{\nu}D^{1+\nu}}\right)^{\frac{1}{\nu}}$. Then, we have
	$$
	\sum_{\ell=0}^{k}\min_{0\leq j\leq \ell}|\theta(x^{j})|\left(\dfrac{\min_{0\leq j\leq \ell}|\theta(x^{j})|}{2M_{\nu}D^{1+\nu}}\right)^{\frac{1}{\nu}}\leq 4(f_{0}^{\max}-f^{\inf}),
	$$
	which implies that
	\begin{equation}
		\min_{0\leq j\leq k}|\theta(x^{j})|\leq\left(\dfrac{2^{2+\frac{1}{\nu}M_{\nu}^{\frac{1}{\nu}}D^{\frac{1+\nu}{\nu}}}(f_{0}^{\max}-f^{\inf})}{k+1-\tilde{k}_{0}}\right)^{\frac{\nu}{1+\nu}}.
	\end{equation}
	Hence, the proof is complete.
	
\end{proof}

\begin{theorem}\label{alg2_thm3}
	Assume that $H$ is convex. Let $\{x^{k}\}$ be the sequence generated by Algorithm {\rm\ref{algo2}}. Assume that there exists $x^{*}\in{\rm dom}(G)$ such that $F(x^{*})\preceq F(x^{k})$ for all $k$. 
	Then, we have 
	$$\delta_{k}(x^{*})\leq \varGamma_{k-\tilde{k}_{0}}$$
	for $k\geq k_{0}+\tilde{k}_{0}$, where 
	$$\varGamma_{k}=\left((\delta_{k_{0}+\tilde{k}_{0}}(x^*))^{-\frac{1}{\nu}}+\dfrac{k-k_{0}}{4A\nu}\right)^{-\nu}, \quad \tilde{k}_{0}=\left\lceil \left(\log_{2}\dfrac{L_{-1}}{\tilde{L}_{0}}\right)_{+}\right\rceil,\quad k_{0}=\left\lceil 4 \left(\log\dfrac{4\delta_{\tilde{k}_{0}}(x^{*})}{A^{\nu}}\right)_{+}\right\rceil$$
	and
	$A=(2M_{\nu})^{\frac{1}{\nu}}D^{\frac{1+\nu}{\nu}}$.
	Furthermore,
	$$
	\min_{0\leq j\leq k}|\theta(x^{k})|\leq e^{\frac{1}{e}}\varGamma_{\left\lfloor\frac{k-\tilde{k}_{0}+k_{0}+1}{2}\right\rfloor}
	$$
	for all $k\geq k_{0}+\tilde{k}_{0}+8A/(\delta_{k_{0}+\tilde{k}_{0}}(x^{*}))^{1/\nu}$.
\end{theorem}

\begin{proof}
	From the relation \eqref{alg2_thm1_eq1}, we have
	\begin{equation*}
		\begin{aligned}
					\min_{i\in\langle m\rangle}(F_{i}(x^{k+1})-F_{i}(x^{*}))&\leq\min_{i\in\langle m\rangle}(F(x^{k})-F_{i}(x^{*}))\\
					&\quad-\dfrac{\min_{0\leq j\leq k}|\theta(x^{j})|}{4}\min\left\{1,\left(\dfrac{\min_{0\leq j\leq k}|\theta(x^{j})|}{2M_{\nu}D^{1+\nu}}\right)^{\frac{1}{\nu}}\right\},
		\end{aligned}
	\end{equation*}
	i.e.,
	\begin{equation}\label{thm3_eq1}
		\delta_{k+1}(x^{*})\leq\delta_{k}(x^{*})-\dfrac{\min_{0\leq j\leq k}|\theta(x^{j})|}{4}\min\left\{1,\left(\dfrac{\min_{0\leq j\leq k}|\theta(x^{j})|}{2M_{\nu}D^{1+\nu}}\right)^{\frac{1}{\nu}}\right\}.
	\end{equation}
	On the other hand, it follows from the convexity of $H$  that $h_{i}(x^{*})-h_{i}(x^{k})\geq\langle \nabla h_{i}(x^{*}),x^{*}-x^{k}\rangle$. Observe that $F(x^{*})\preceq F(x^{k})$ for all $k$, we get
	$
		0\geq f_{i}(x^{*})-f_{i}(x^{k})\geq\langle\nabla h_{i}(x^{k}),x^{*}-x^{k}\rangle+g_{i}(x^{*})-g_{i}(x^{k}),
	$
	for all $i\in\langle m\rangle$, which implies that
	\begin{equation*}
		\begin{aligned}
			0&\geq \max_{i\in\langle m\rangle}
			(f_{i}(x^{*})-f_{i}(x^{k}))\\
			&\geq\max_{i\in\langle m\rangle}\{\langle\nabla h_{i}(x^{k}),x^{*}-x^{k}\rangle+g_{i}(x^{*})-g_{i}(x^{k})\}\\
			&\geq\max_{i\in\langle m\rangle}\{\langle\nabla h_{i}(x^{k}),s(x^{k})-x^{k}\rangle+g_{i}(s(x^{k}))-g_{i}(x^{k})\}\\
			&=\theta(x^{k}),
		\end{aligned}
	\end{equation*}
	where the third inequality follows from the optimality of $s(x^{k})$ in \eqref{algo1_eq1}. Therefore,
	\begin{equation*}
		0\leq\min_{i\in\langle m\rangle}
		(f_{i}(x^{k})-f_{i}(x^{*}))=\delta_{k}(x^{*})\leq|\theta(x^{k})|.
	\end{equation*}
	By Lemma \ref{lem_aux1} with $\gamma_{k}=\delta_{k+\tilde{k}_{0}}(x^{*})$, $\beta_{k}=\min_{0\leq j\leq k+\tilde{k}_{0}}|\theta(x^{j})|$, $\alpha=1/\nu$, $c=1/4$ and $A=(2M_{\nu})^{\frac{1}{\nu}}D^{\frac{1+\nu}{\nu}}$, we can obtain the desired results.
\end{proof}

\section{Numerical experiments}\label{num_exp}

In this section, we present some numerical experiments to illustrate the performance of the proposed algorithms, PGM-CondG (see Algorithm \ref{algo1}) and FGM-CondG (see Algorithm \ref{algo2}). All codes are written in MATLAB R2020b and run on PC with the specification: Processor Intel i7-10700 and 2.90 GHz, 32.00 GB RAM. 

For FGM-CondG, the initial parameter $L_{-1}$ is set to 1. The success stopping criterion for all algorithms is defined as $|\theta(x^{k})|\leq\epsilon$, where $\epsilon=10^{-4}$. The maximum number of allowed outer iterations is set to 1000, after which the algorithm is considered to have failed.

For test problems, we choose the differentiable part $H$ involved in the problem \eqref{mop} as in Appendix A, and we consider the following two special cases regarding $G(x)=(g_{1}(x),...,g_{m}(x))^{\top}$ involved in the problem \eqref{mop}:
\begin{enumerate}[label=\textup{(\roman*)}]
	\item $G$ is an indicator function of a convex set $\Omega\subseteq\mathbb{R}^{n}$. In this case, the subproblem in Algorithms \ref{algo1} and \ref{algo2} can be equivalently transfromed as the following optimization problem:
	\begin{equation}\label{sub_problem1}
		\begin{aligned}
			\mathop{\text{min}}_{\tau,u}\quad &\tau\\
			\text{s.t.}\quad&\langle \nabla h_{i}(x^{k}),u-x^{k} \rangle\leq \tau,\quad i\in\langle m\rangle,\\
			&u\in\Omega.
		\end{aligned}
	\end{equation}
	\item For all $i\in\langle m\rangle$, let
	\begin{equation}\label{maxg}
		g_{i}(x)=\max_{z\in\mathcal{Z}_{i}}\langle x,z\rangle,
	\end{equation} 
	where $\mathcal{Z}_{i}\subseteq \mathbb{R}^{n}$ is the uncertainty set. We assume that ${\rm dom}(g_{i})=\{x\in\mathbb{R}^{n}:x^{L}\preceq x\preceq x^{U}\}$ for all $i\in\langle m\rangle$, where $x^{L},x^{U}\in\mathbb{R}^{n}$ are given in Appendix A.	
	In the experiments, $\mathcal{Z}_{i}$ is assumed to be a polytope, i.e., $\mathcal{Z}_{i}=\{z\in\mathbb{R}^{n}:-\delta e\preceq B_{i}z\preceq\delta e\}$, where $B_{i}\in\mathbb{R}^{n\times n}$ is a given non-singular matrix and $\delta>0$ is a given constant. In our experiments, the elements of $B_{i}$ are randomly chosen between 0 and 1, and $\delta$ is randomly chosen from the interval $[0.01,0.1]$. 	Let $C_{i}=[B_{i};-B_{i}]\in\mathbb{R}^{2n\times n}$ and $b_{i}=\delta e\in\mathbb{R}^{2n}$. According to the descriptions in \cite{assunccao2023generalized}, \eqref{maxg} can be rewritten as
	\begin{equation}\label{maxg_form1}
		\begin{aligned}
			\mathop{\text{max}}_{z}\quad &\langle x,z\rangle\\
			\text{s.t.}\quad& C_{i}z\preceq b_{i},
		\end{aligned}
	\end{equation}
	and further, the subproblem in Algorithms \ref{algo1} and \ref{algo2} is  equivalent to the following linear optimization problem:	
	\begin{equation}\label{sub_problem2}
		\begin{aligned}
			\mathop{\text{min}}_{\tau,u,w_{i}}\quad &\tau\\
			\text{s.t.}\quad&\langle b_{i},w\rangle-g_{i}(x^{k})+\langle\nabla h_{i}(x^{k}),u-x^{k}\rangle\leq\tau,\\
			&C_{i}^{\top}w_{i}=u,\\
			&w_{i}\succeq 0,\quad i\in\langle m\rangle,\\
			&x^{L}\preceq u\preceq x^{U}.
		\end{aligned}
	\end{equation}
\end{enumerate}

In our experiments, the standard MATLAB subroutine \texttt{linprog} is adopted to solve the problems  \eqref{sub_problem1}, and \eqref{maxg_form1} and \eqref{sub_problem2}. For each algorithm, we solve all problems 100 times by using a uniform random distribution of starting points. We compare these methods using as the performance measurement: the median number of iterations (Iter), the median number of function evaluations (Feval) and the median computational time (in seconds) to reach the Pareto critical point from an initial point (CPU). We also use the well-known \emph{Purity} and ($\Gamma$ and $\Delta$) \emph{Spread} metrics \cite{custodio2011direct} to quantitatively evaluate the algorithm's ability to generate the Pareto front.


Table \ref{tab1} presents the numerical results of the PGM-CondG and FGM-CondG algorithms on the selected problem. Note that the PGM-CondG algorithm does not involve the evaluation of functions, so ``0'' is used to fill in the fourth column of Table \ref{tab1}. ``Failed'' referes to the algorithm attained the maximum number of iterations without arriving at a solution (recall that the maximum number of iterations is 1000). As observed, for Cases i and ii, the FGM-CondG algorithm outperforms the PGM-CondG algorithm in terms of the number of iterations on the selected test problems. The efficiencies of the PGM-CondG and FGM-CondG algorithms are respectively 28.6\% and 71.4\%, considering the CPU time as the performance measurement.

\begin{table}\footnotesize
	\centering
	\caption{Numerical results of the PGM-CondG and FGM-CondG algorithms on the chosen set of test problems.}
     \begin{tabular}{llllllllll}
 	\toprule
 	\multicolumn{2}{l}{Problem} &       & \multicolumn{3}{l}{PGM-CondG} &       & \multicolumn{3}{l}{PGM-CondG} \\
 	\cmidrule{1-2}\cmidrule{4-6}\cmidrule{8-10}    $H$ function & $G$ function &       & Iter  & Feval & CPU   &       & Iter  & Feval & CPU \\
 	\midrule
 	BK1   & Case i &       & \textBF{2}     & 0     & \textBF{0.0167} &       & \textBF{2}     & 5     & 0.0189 \\
 	& Case ii &       & \textBF{4}     & 0     & \textBF{0.0473} &       & \textBF{4}     & 7     & 0.0498 \\
 	IKK1  & Case i &       & 7     & 0     & \textBF{0.0238} &       & \textBF{5}     & 10    & 0.0443 \\
 	& Case ii &       & 98    & 0     & 1.6224 &       & \textBF{9}     & 11    &\textBF{0.1097 }\\
 	IM1   & Case i &       & 3     & 0     & \textBF{0.0112} &       & \textBF{2 }    & 3     & 0.0159 \\
 	& Case ii &       & \textBF{2 }    & 0     & 0.0268 &       & \textBF{2}     & 3     & \textBF{0.0192} \\
 	JOS1  & Case i &       & 335   & 0     & 1.0969 &       & \textBF{46}    & 91    & \textBF{0.2502} \\
 	& Case ii &       & 271   & 0     & 3.6298 &       & \textBF{118}   & 235   & \textBF{1.7312} \\
 	Lov1  & Case i &       & 5     & 0     & \textBF{0.0178} &       & \textBF{4}     & 9     & 0.0296 \\
 	& Case ii &       & 5     & 0     & 0.0669 &       & \textBF{4}     & 8.5   & \textBF{0.0547} \\
 	MAN1  & Case i &       & 164   & 0     & 1.4579 &       & \textBF{8}     & 18    & \textBF{0.1111} \\
 	& Case ii &       & Failed & 0     & 17.4920 &       & \textBF{7}     & 17    & \textBF{0.1891} \\
 	MAN2  & Case i &       & 20    & 0     & 0.2054 &       & \textBF{6}     & 13    & \textBF{0.1122} \\
 	& Case ii &       & 130   & 0     & 2.3900 &       & \textBF{6}     & 13.5  & \textBF{0.1817} \\
 	MAN3  & Case i &       & \textBF{3}     & 0     & \textBF{0.0322} &       & \textBF{3}     & 6     & 0.0555 \\
 	& Case ii &       & \textBF{3}     & 0     & 0.0754 &       & \textBF{3}     & 6     & \textBF{0.0738} \\
 	MGH33 & Case i &       & 11    & 0     & 0.0805 &       & \textBF{2}     & 13.5  & \textBF{0.0231} \\
 	& Case ii &       & Failed & 0     & 28.0880 &       & \textBF{14}    & 24.5  & \textBF{0.5836} \\
 	MHHM2 & Case i &       & \textBF{2}     & 0     & \textBF{0.0118} &       & \textBF{2}     & 5     & 0.0197 \\
 	& Case ii &       & \textBF{3}     & 0     & \textBF{0.0455} &       & \textBF{3}     & 6.5   & 0.0551 \\
 	SP1   & Case i &       & 86    & 0     & 0.2723 &       & \textBF{13}    & 26    & \textBF{0.0786} \\
 	& Case ii &       & 83.5  & 0     & 1.0457 &       & \textBF{12.5}  & 26.5  & \textBF{0.1883} \\
 	Toi8  & Case i &       & 24.5  & 0     & 0.0846 &       & \textBF{7 }    & 16    & \textBF{0.0499} \\
 	& Case ii &       & 146   & 0     & 2.3281 &       & \textBF{10}    & 19    & \textBF{0.1753} \\
 	VU1   & Case i &       & 303.5 & 0     & 1.0314 &       & \textBF{155}   & 311.5 & \textBF{0.9673 }\\
 	& Case ii &       & 93    & 0     & 1.1784 &       &\textBF{32.5}  & 66    & \textBF{0.4588} \\
 	VU2   & Case i &       & 6     & 0     & 0.0229 &       & \textBF{3}     & 4     & \textBF{0.0196} \\
 	& Case ii &       & 7     & 0     & 0.0912 &       & \textBF{3}     & 4     & \textBF{0.0275 }\\
 	\bottomrule
 \end{tabular}%
	\label{tab1}%
\end{table}%

The performance profiles of the Purity and Spread metrics are shown in Figs. \ref{noncompositemetric} and \ref{compositemetric}. As
can be seen, when $G$ corresponds to Case i, the FGM-CondG algorithm outperforms the PGM-CondG algorithm for the three metrics. When $G$ corresponds to Case ii, the FGM-CondG algorithm shows superior performance over the PGM-CondG algorithm for the Purity and $\Delta$ metrics, while no signifcant diference is noticed for the $\Gamma$ metric.

\begin{figure}
	\centering
	\includegraphics[width=1\linewidth]{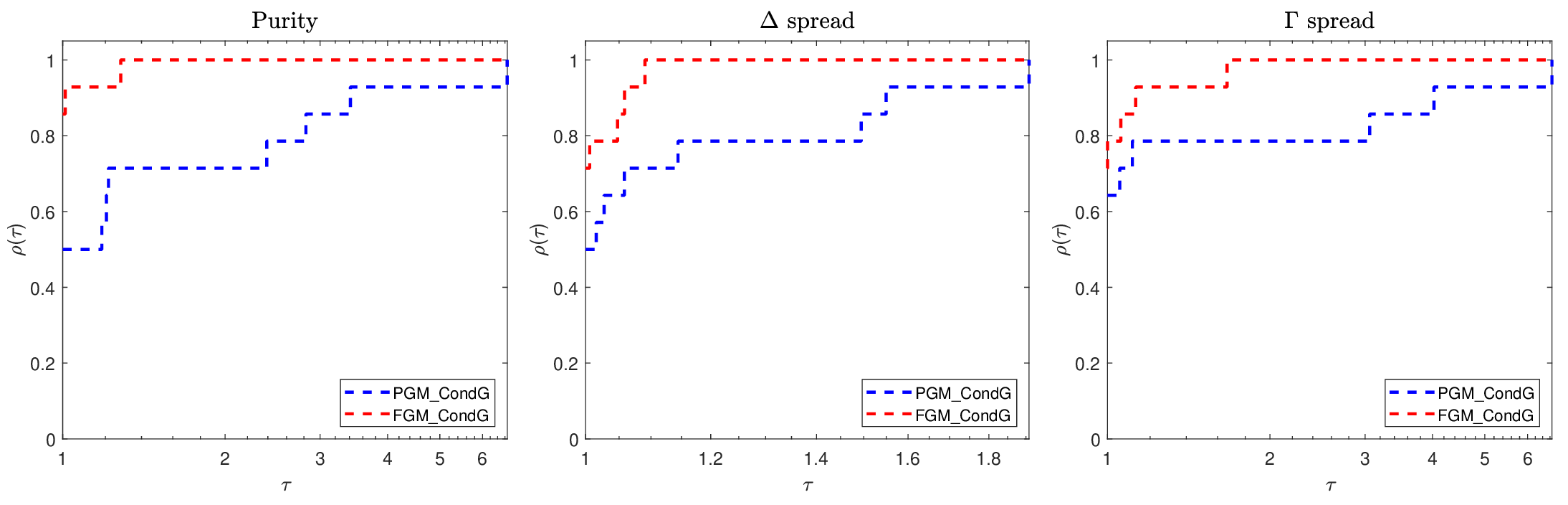}
	\caption{Purity and Spread performance profiles of PGM-CondG and FGM-CondG for the case where $G$ corresponds to Case i.}
	\label{noncompositemetric}
\end{figure}

\begin{figure}
	\centering
	\includegraphics[width=1\linewidth]{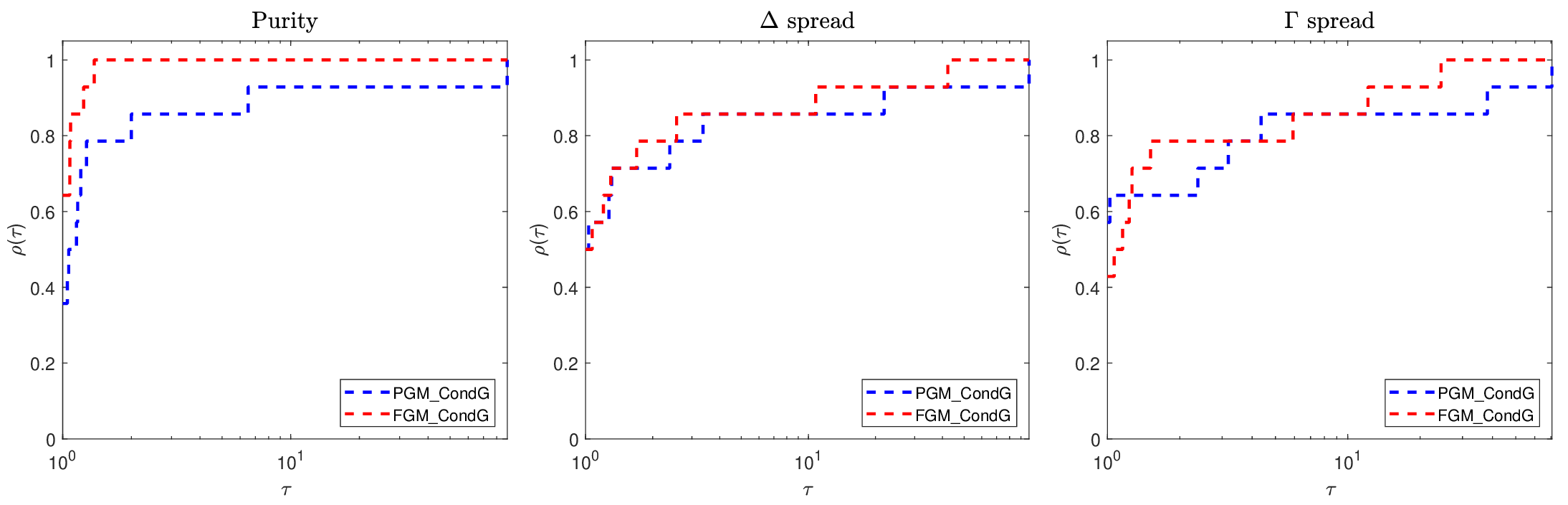}
	\caption{Purity and Spread performance profiles of PGM-CondG and FGM-CondG for the case where $G$ corresponds to Case ii.}
	\label{compositemetric}
\end{figure}

From the numerical experiments, we conclude that the use of an adaptive line search process in Algorithm \ref{algo2} not only avoids the need to compute the parameters $\nu$ and $M_{\nu}$, but also significantly improves computational efficiency and solution quality.

\section{Conclusions}\label{conclusion}

In this paper, we first introduced a parameter-dependent generalized multiobjective conditional gradient method for solving multiobjective composite optimization problems, where each objective function is expressed as the sum of two functions, one of which is assumed to be continuously differentiable. In this approach, the step size explicitly depends on the problem parameters. We then developed a novel parameter-free generalized multiobjective  conditional gradient method for the same class of problems, which eliminates the need for prior knowledge of problem parameters. The convergence properties of both methods were established under mild conditions. Finally, the performance of the proposed methods was evaluated on several test problems.



\section*{Appendix A}

Here, we list the test problems (the continuous differentiable components $h_{i}$, $i\in\langle m\rangle$) used in Section \ref{num_exp}. 

\begin{enumerate}
	\item \textbf{MAN}: 
	\begin{align*}
		h_{1}(x) = \frac{1}{p}\|Q_{1}x-b_{1}\|_{p}^{p}, \quad
		h_{2}(x) = \frac{1}{p}\|Q_{2}x-b_{2}\|_{p}^{p}, 
	\end{align*}
	where $Q_{1},Q_{2}\in\mathbb{R}^{n\times n}$ are diagonal matrices whose diagonal entries are randomly generated according to the uniform distribution over $[1,2]$, $b_{1},b_{1}\in\mathbb{R}^{n}$ and $p\in(1,2]$. 
	In the experiments, we set	$m=n=2$,
	$b_{1}=(-0.6,-0.6)^{\top}$, $b_{2}=(-0.5,-0.5)^{\top}$, $x^{L}=(-1,-1)^{\top}$, $x^{U}=(1,1)^{\top}$, $Q_{1}=Q_{2}=I$ (the identity matrix) and $p\in\{1.3,1.6,2\}$. MAN1, MAN2 and MAN3 indicate that the parameters in MAN are set to $p=1.3,1.6$ and 2, respectively

	\item The rest test problems are provided in Table \ref{appendix_pro}.
	\begin{table}[H]\small
		\centering
		\centering
		\caption{Test problems.}
		\setlength{\tabcolsep}{4mm}{
			\begin{tabular}{llllll}
				\hline
				Problem & $n$     &$ m $    & $x_{L}$     & $x_{U}$     & Source \\\hline
				BK1   & 2     & 2     & $(-5,-5)$ & $(10,10)$ & \cite{huband2006review} \\
				IKK1   & 2     & 3    & $(-50,-50) $& $(50,50)$ & \cite{huband2006review} \\
				IM1   & 2     & 2    & $(1,1) $& $(4,2)$ & \cite{huband2006review} \\
				JOS1	& 10  & 2     & $(-100,\ldots,-100) $&$ (100,\ldots,100) $& \cite{jin2001dynamic}  \\
				Lov1  & 2     & 2     &$ (-10,-10) $&$ (10,10)$ & \cite{lovison2011singular} \\
				MHHM2  & 2     & 3     &$ (0,0)$ & $(1,1)$ & \cite{huband2006review} \\
				MGH33 & 10     & 10     & $(-1,\ldots-1)$ & $(1,\ldots,1)$ &\cite{more1981testing}   \\
				SP1   & 2     & 2     &$ (-100,-100)$ & $(100,100)$ & \cite{huband2006review} \\
				Toi8   & 3     & 3     & $(-1,-1,-1) $& $(1,1,1)$ & \cite{toint1983test} \\
				VU1 & 2     & 2     & $(-3,-3)$ & $(3,3)$ &\cite{huband2006review}    \\
				VU2 & 2     & 2     & $(-3,-3)$ & $(3,3)$ &\cite{huband2006review}    \\
				\hline
		\end{tabular}}%
		\label{appendix_pro}%
	\end{table}%
\end{enumerate}


\bibliographystyle{abbrv}
\bibliography{myrefs}

\end{document}